\documentclass{article}
\usepackage[journal=JST, lang=american]{ems-journal}

\newtheorem{theorem}{Theorem}[section]

\newtheorem{proposition}[theorem]{Proposition}

\newtheorem{lemma}[theorem]{Lemma}

\theoremstyle{definition}
\newtheorem{remark}[theorem]{Remark}

\def\K{\mathfrak{K}}

\def\T{\mathcal{P}}
\def\P{{\rm Prob}}

\def\R{\mathbb{R}}

\def\C{{\mathbb C}}

\def\X{\mathcal X}

\def\bb1{{\rm{1}\hspace{-3pt}\mathbf{l}}}
\def\Ie0{[-\epsilon_0,\epsilon_0]}

\def\BC2{\mathbb{B}\big(\mathbb{C}^2\big)}

\def\beq{\begin{equation}}
\def\eeq{\end{equation}}

\newcommand{\virg}[1]{``#1''}

\numberwithin{equation}{section}

\begin{document}

\title{Spectral and dynamical results related to certain non-integer base expansions on the unit interval}
\titlemark{Non-integer base expansions on the unit interval}




\emsauthor{1}{
	\givenname{Horia D.}
	\surname{Cornean}
	\mrid{633226}
	\orcid{0000-0003-2700-8785}}{H.~D.~Cornean}
\emsauthor{2}{
	\givenname{Ira W.}
	\surname{Herbst}
	\mrid{84635}
	\orcid{}}{I.~W.~Herbst}
\emsauthor{3}{
	\givenname{Giovanna}
	\surname{Marcelli}
	\mrid{1304172}
	\orcid{0000-0002-5684-5016}}{G.~Marcelli}

\Emsaffil{1}{
	\department{Department of Mathematical Sciences}
	\organisation{Aalborg University}
	\rorid{04m5j1k67}
	\address{Thomas Manns Vej 23}
	\zip{9220}
	\city{Aalborg}
	\country{Denmark}
	\affemail{cornean@math.aau.dk}}
\Emsaffil{2}{
	\department{Department of Mathematics}
	\organisation{University of Virginia}
	\rorid{0153tk833}
	\address{141 Cabell Drive, Kerchof Hall}
	\zip{VA 22904}
	\city{Charlottesville}
	\country{USA}
	\affemail{iwh@virginia.edu}}
\Emsaffil{3}{
	\department{Dipartimento di Matematica e Fisica}
	\organisation{Università di Roma Tre}
	\rorid{05vf0dg29}
	\address{L.go S. L. Murialdo 1}
	\zip{00146}
	\city{Roma}
	\country{Italy}
	\affemail{giovanna.marcelli@uniroma3.it}}

\classification[37A50]{37A30}


\begin{abstract}
We consider certain non-integer base $\beta$-expansions of Parry's type and
we study various properties of the transfer (Perron-Frobenius) operator $\T\colon L^p([0,1])\to L^p([0,1])$ with $p\geq 1$  and  its associated composition  (Koopman) operator, which are induced by a discrete dynamical system on the unit interval related to these $\beta$-expansions.
 
 We show that if $f$ is Lipschitz, then the iterated sequence $\{\T^N f\}_{N\geq 1}$ converges exponentially fast (in the $L^1$ norm) to an invariant state corresponding to the eigenvalue $1$ of $\T$. This \virg{attracting} eigenvalue is not isolated: for  $1\leq p\leq 2$ we show that the point spectrum of $\T$ also contains the whole open complex unit disk and we explicitly construct an eigenfunction for every $z$ with $|z|<1$.
\end{abstract}

\maketitle


\section{Introduction and main results}
Let us fix two integers $n\geq 2$ and $q\geq 1$.  There exists a unique positive number (see Lemma \ref{lemmagh2}) 
\begin{equation}
\label{eqn:beta}   
\beta_{n,q}\equiv \beta\in (q,q+1)
\end{equation}
which obeys the following equation: 
\begin{equation}\label{1}
 1=\frac{q}{\beta}+\frac{q}{\beta^2}+\cdots +\frac{q}{\beta^n}.
\end{equation}
We consider representations of real numbers in non-integer base $\beta$ of the type \eqref{1}, which are called $\beta$-expansions.
Expansions in non-integer bases were firstly introduced by the seminal work of R{\'e}nyi \cite{Re}, as a generalization of the standard integer
base expansions. The original method to determine the \virg{digits} is the greedy algorithm \cite{Re,Pa, Pe}, which is tightly connected to the study of the map 
\begin{equation}
\label{eqn:Tbeta}
T_\beta:[0,1)\mapsto [0,1),\qquad T_\beta(x)=\beta x-\lfloor \beta x\rfloor,
\end{equation}
see Appendix \ref{app:gral} for some of its basic properties. Without putting certain restrictions on the coefficients, such expansions are far from being unique (see \cite{KLP} and references therein).
Such expansions are also related to symbolic dynamics \cite{Bl, Pa, Re}, which is not the  main focus of the current paper.

 We are mostly interested in the investigation of certain spectral and dynamical properties of the transfer (or Perron-Frobenius) operator $\T:L^p([0,1])\mapsto L^p([0,1])$ with $p\geq 1$, and its associated composition (or Koopman) operator $\K$, which are induced by the above map $T_\beta$ \cite{LY, Suz, Wal}.

In general, the transfer operator $\T$ describes the discrete time evolution of certain probability densities associated to some stochastic variables, evolution related to the iteration of a certain map, in our case $T_\beta$ \cite{BG,CCD,Go}. More specific details about these objects will be given in the subsequent part of the introduction, where we will also formulate our main results: Theorem \ref{thm1} and \ref{thm2}. There, it is stated that if $f$ is Lipschitz, then the iterates $\T^N f$ converge exponentially fast (in the $L^1$ norm and $N\to\infty$) to an invariant state corresponding to the eigenvalue $1$ of $\T$. On the other hand, the eigenvalue $1$ is far from being isolated: if $1\leq p\leq 2$ we show that the point spectrum of $\T$ also contains the open complex unit disk; namely, for every $|z|<1$ and we explicitly construct a corresponding $\psi_z$ such that $\T \psi_z=z\psi_z$.

\subsection{The transfer operator}
Let us assume that $X:\Omega \mapsto [0,1]$ is an absolutely continuous stochastic variable with a probability density
function (PDF) denoted by $f\in L^1([0,1])$. More precisely: for every $x\geq 0$
$$\P \big (X\leq x\big ):=\int_0^x\, f(t)\, dt.$$
Any number $X(\omega)\in (0,1)$ has a well defined \emph{\virg{greedy} decomposition} of the type (see Lemma \ref{prop10}) 
$$X(\omega)=\sum_{k\geq 1} X_k(\omega)\, \beta^{-k},\quad X_k(\omega)\in \{0,1, \dots, q\}.$$
 The first coefficient $X_1$ defines a discrete stochastic variable $X_1:\Omega \mapsto \{0,\dots, q\}$, where (remember that $q<\beta<q+1$)  
$$X_1(\omega):= j\in \{0,\dots, q\} \quad \text{whenever}\quad j/\beta \leq X(\omega)<(j+1)/\beta
$$
which implies
$$\P(X_1=j)=\P\big (j/\beta\leq X <(j+1)/\beta\big ),\quad 0\leq j\leq q.$$

Assuming that $f(x)=0$ if $x\not\in [0,1]$, then the new stochastic variable $\tilde{X}=\beta(X-X_1/\beta)$  is also absolutely continuous and has a PDF (denoted by $\T f$) which equals: 
 \begin{align}\label{3}
  (\T f)(x)=\beta^{-1}\sum_{j=0}^{q} f\big ( (j+x)/\beta \big ).
\end{align}

Formula \eqref{3} is due to the fact that for $x\geq 0$ we have:
\begin{align*}
\P\big (\beta (X- X_1/\beta ) \leq x\big )=\P\big ( X\leq (X_1+x)/\beta\big )=\sum_{j=0}^q \P\big ( j/\beta \leq X\leq  (j+x)/\beta \big ),
\end{align*}
which we then differentiate with respect to $x$.

In order to formulate our first theorem, we need to state the following result, which goes back to \cite{Pa}: 

\begin{proposition}
\label{prop:u1}
There exists a piecewise constant function $u_1$ which is positive a.e.\ with $\int_0^1 u_1(x)\, dx =1$ such that $\T u_1=u_1$. 
\end{proposition}

Our first main theorem is as follows:
\begin{theorem}\label{thm1} 
Let $n\geq 2$ and $q\geq 1$ be two integers. Let $\T\equiv\T_\beta\colon L^1([0,1])\to L^1([0,1])$ be defined as in \eqref{3}, where $\beta\equiv\beta_{n,q}$ is introduced in \eqref{eqn:beta}.
Then there exist two constants $K_1(n,q)\geq 0 $ and $K_2(n,q)\geq 1/2$ such that for every  Lipschitz function $f$ with $|f(x) - f(y)| \le L_f|x-y|$ we have 
\begin{equation}
\label{eqn:rateconv}
\left \Vert \T^N f-u_1 \int_0^1f(t)dt\right \Vert_{L^1}\leq \,  K_1 \, \big (L_f +{\Vert f\Vert}_{L^\infty} \big )\, \beta^{-K_2\, N}\qquad\forall N\geq 1.
\end{equation}
If $n=2$ we have 
\begin{equation}\label{aug1}
\beta=\frac{q+\sqrt{q^2+4q}}{2},\quad K_2=\frac{2-\ln(q)/\ln(\beta)}{3-\ln(q)/\ln(\beta)}\, .
\end{equation}
\end{theorem}

\begin{remark}
We have a few extra comments: 
\begin{enumerate}[label=(\roman*), ref=(\roman*)]
\item \label{it:pointvi} By using that the map $\T$ is non-expansive on $L^1$ (see \eqref{eqn:Pnonexp}), a density argument implies that if $f\in L^1([0,1])$, then $$\lim_{N\to \infty}\left \Vert \T^N f-u_1 \int_0^1f(t)dt\right \Vert_{L^1}=0 .$$ 

\item \label{it:uniqueu1}
Point (i) implies that the function $u_1$ constructed in  Proposition \ref{prop:u1} is, up to a constant factor, the unique $L^1$ eigenfunction of $\T$ corresponding to the eigenvalue $1$. We note that Parry \cite{Pa} also obtained an explicit formula for $u_1$ in an even more general case.  For $q=1$ (see \eqref{1}), an exponential decay in sup norm with the same exponent as ours has been previously obtained in \cite{HMS2}, but using a slightly different approach (we will explain it in a moment) and with a very different method concerning the convergence. Namely, let  
$$X=\sum_{k=1}^\infty X_k \,\beta^{-k}$$
be the $\beta$-expansion (with $q=1$) of an absolutely continuous random variable
$X$ on the unit interval. Then \cite{HMS2} analyses the convergence rate of the PDF of the scaled remainder $\sum_{k=1}^\infty X_{m+k} \beta^{-k}$ when $m$ tends to infinity to the asymptotic distribution $u_1$. If the density of $X$ is $f$, then $\T^m f$ is nothing but the density associated with the above scaled remainder.

\item  \label{it:LY}
In \cite{LY} it is shown the existence of a C{\'e}saro limit $\frac{1}{N}\sum_{k=1}^N \mathcal{P}^k f$ in the $L^1$-norm for the more general case of piecewise monotonic and expanding maps.

\item \label{it:erg} We now briefly outline some consequences for the ergodicity properties \cite{DK} of the map $T_\beta$ in \eqref{eqn:Tbeta}. It is measure preserving on $[0,1]$ equipped with the measure density $u_1$.
We consider stochastic variables of the type $F:[0,1]\mapsto \R$ with
$${\rm Prob}(F\in (c,d)):=\int_{F^{-1}((c,d))} u_1(x)dx,\quad \forall \, c<d. $$
For every integer $k\geq 0$ we define $\X_k:[0,1]\mapsto \R$  given by 
$$\X_k(x):=g(T_\beta^k(x)),
$$
for some $g\in L^p([0,1])$ with $1\leq p \leq \infty$. If $g$ is Lipschitz, by using Theorem \ref{thm1} one can prove that these random variables have the same mean value and exponentially decaying correlations, which in turn implies \cite[Theorem 1]{AA} the strong law of large numbers. 
\end{enumerate}
\end{remark}

The proof of Theorem \ref{thm1} is given in Section \ref{sect2}. 

\subsection{The composition (Koopman) operator}

Let us recall the definition of $T_\beta:[0,1)\mapsto [0,1)$ given by:  
\begin{equation}\label{march1}
T_\beta(x)=\beta x-\lfloor \beta x\rfloor= \beta x-j,\quad j/\beta\leq x<(j+1)/\beta, \quad x\in [0,1), \quad j\in\{0,1,\dots,q\}\\
.\end{equation}
We define the operator 
\begin{equation}\label{may1}
\K :L^p([0,1])\mapsto L^p([0,1]),\quad (\K \, g)(x):=g\big (T_\beta(x)\big ),\quad 1\leq p\leq \infty.
\end{equation}
We may also consider the operator $\T$ from \eqref{3} acting on $L^{p'}([0,1])$ to itself with $1/p+1/p'=1$ and $1\leq p'\leq \infty$. Then if $f\in L^{p'}([0,1])$ and $g\in L^{p}([0,1])$ we have 
\begin{equation}\label{oct1}
\begin{aligned}
\int_0^1 \overline{f(t)}\, (\K \, g)(t)\, dt 
&=\sum_{j=0}^{q-1}\int_{j/\beta}^{(j+1)/\beta} \overline{f( t)}\, g(\beta t-j) \, dt 
+\int_{q/\beta}^1 \overline{f( t)}\, g(\beta t-q)\, dt\\
&= \int_0^1 \overline{[\T f](x)}\, g(x)\, dx,
\end{aligned}
\end{equation}
where in the last equality we used that $f(x)=0$ when $x>1$. 

The main spectral results of this paper are contained in the next theorem.

\begin{theorem}\label{thm2}  The following properties hold:
\begin{enumerate}[label=(\roman*), ref=(\roman*)]
\item \label{itT2:1} Define the numbers 
$$x_j:=q\beta^{-2}+\cdots + q\beta^{-n}+j/\beta,\quad 0\leq j\leq q .$$
They obey $j/\beta<x_j<(j+1)/\beta$ when $0\leq j\leq q-1$, and $x_q=1$.

If $q=1$ we define
$$\psi_0(t)=\left\{  
\begin{matrix}
e^{\pi i \beta t} & \text{if} & j/\beta\leq t<x_j, & 0\leq j\leq 1 \\
0 & \text{if} & x_0\leq t<1/\beta &  
\end{matrix}
\right . \; .$$

If $q>1$ we define 
$$\psi_0(t)=\left\{  
\begin{matrix}
e^{2\pi i \beta t/(q+1)} & \text{if} & j/\beta\leq t<x_j, & 0\leq j\leq q \\
e^{2\pi i \beta t/q} & \text{if} & x_j\leq t<(j+1)/\beta, & 0\leq j\leq q-1 
\end{matrix}
\right . \; .$$
Then $\psi_0\in L^\infty $ and $\T \psi_0=0$ a.e.. Note that when $n\equiv\infty$ then $\beta\equiv q+1$ and $\psi_0(t)\equiv e^{2\pi i t}$. See Figure \ref{fig:2}
for an illustration of the function $\psi_0$ for the cases $q=1$ and $q=3$.
 \item \label{itT2:2} 
The operator $\widetilde{\K}:=u_1^{1/p}\, \K \, u_1^{-1/p}$ is a non-surjective isometry on $L^p([0,1])$ for $1\leq p\leq \infty$. 
\item \label{itT2:3} The spectrum of $\widetilde{\K}$ and $\K$ equals $\overline{\mathbb{D}}=\{z\in \C:\, |z|\leq 1\}$ for $1\leq p\leq \infty$.  

\item \label{itT2:4}  Let $|z|<1$. Then the function $$\psi_z=u_1^{1/2}\Big ({\rm Id} -z\, u_1^{1/2} \K u_1^{-1/2}\Big )^{-1}u_1^{-1/2}\, \psi_0 \in L^2([0,1])\subset L^{p'}([0,1]),\quad 1\leq p'\leq 2,$$
is an eigenfunction of $\T$ which obeys $\T \psi_z =z\, \psi_z$. 
\end{enumerate}
\end{theorem}

The proof of this theorem is given in Section \ref{sect3}. We note that when $\T$ is restricted to functions of bounded variations, its spectrum is quite different \cite{Suz}.  

\begin{figure}
    \centering
    \includegraphics[scale=0.19]{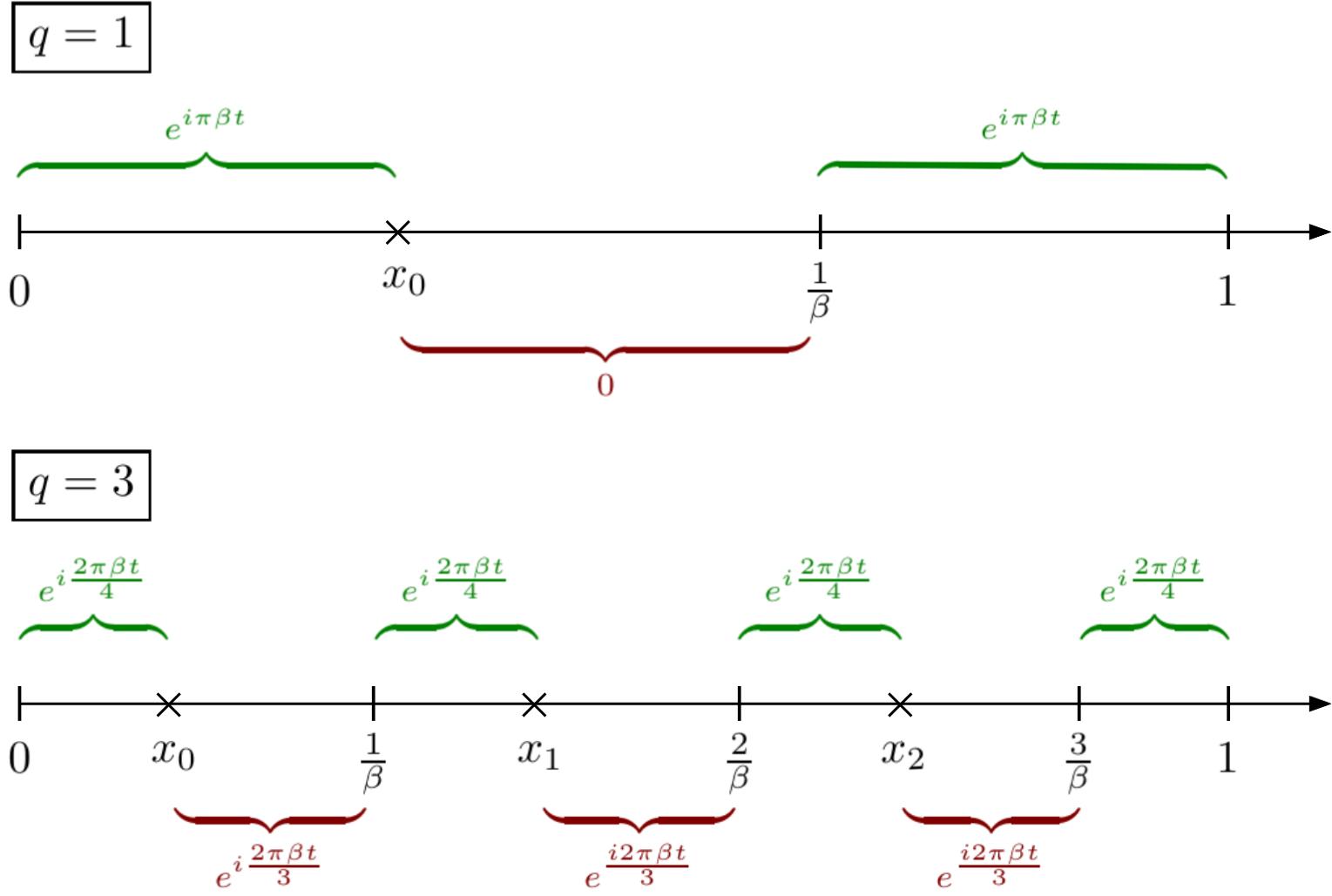}
    \caption{Illustration of the map $\psi_0$}
    \label{fig:2}
\end{figure}

\section{Proof of Theorem \ref{thm1}}\label{sect2}

\subsection{Preliminaries}
Notice that $\T$ maps non-negative functions into non-negative functions and for any function $f\in L^1([0,1])$ we have:
\begin{equation}
\label{eqn:intT}
\int_0^1 (\T f)(x)\, dx=\int_0^1 f(x)\, dx\, .
\end{equation}

Indeed, if $0\leq j\leq q-1$, we have $$[0,1]\ni x\mapsto (j+x)/\beta \in [j/\beta,(j+1)/\beta],$$ hence these intervals cover the interval $[0,q/\beta]$. Also, due to \eqref{1} we have 
$$[0,q/\beta+\cdots +q/\beta^{n-1}]\ni x\mapsto (q+x)/\beta \in [q/\beta,1].$$
Equality \ref{eqn:intT} follows after a change of variable on each interval.
Moreover, this together with $|\T f|\leq \T |f|$ imply that the linear map $\T$ is non-expansive on $L^1$, i.e.\ 
\begin{equation}
\label{eqn:Pnonexp}
\Vert\T f \Vert_{L^1}\leq \Vert f\Vert _{L^1}\;\text{for all $f\in L^1([0,1])$}.
\end{equation}

\subsection{Subdividing the interval $[0,1]$}
\begin{figure}
    \centering
    \includegraphics[scale=0.15]{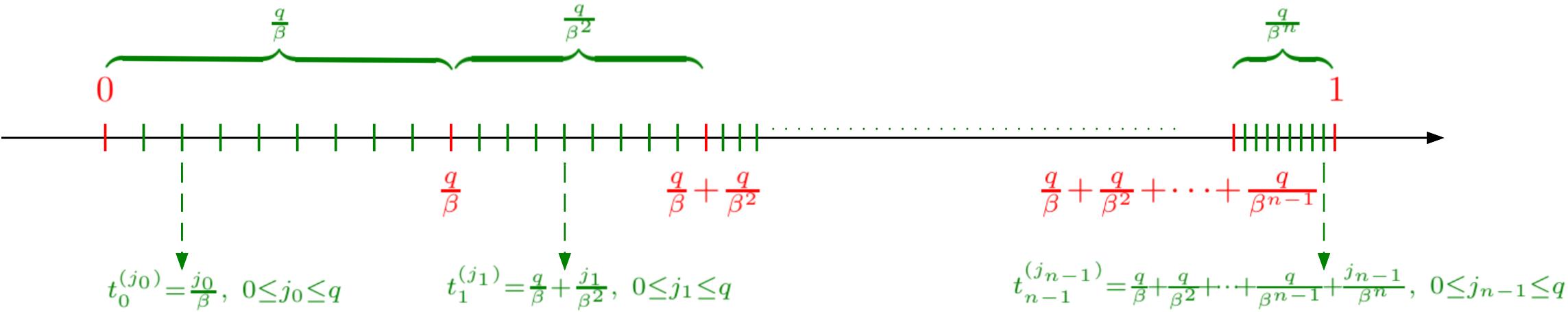}
    \caption{The first layer}
    \label{fig:1}
\end{figure}
In Figure \ref{fig:1} we introduce a decomposition of the interval $[0,1]$, which we 
will explain in what follows. The characteristic functions of the intervals between two consecutive red points will form a generating system, and it is important to know how $\T$ acts on them. This will be done in Lemma \ref{lemmagh1}. 

First, we have the numbers in red given by $0,q/\beta$, $q/\beta+ q/\beta^2$, ..., and $q/\beta+ q/\beta^2+\cdots q/\beta^{n-1},1$.

Second, we want to define the green numbers, which include the red ones, see Figure \ref{fig:1}. Let us start with those between $0$ and $q/\beta$. For $j_0\in \{0,\dots, q\}$ we define the first set of green numbers: $t_0^{(j_0)}=j_0/\beta$, with $t_0^{(q)}=q/\beta$. The distance between two consecutive such numbers is $1/\beta$.   

The green numbers between $q/\beta$ and $q/\beta+q/\beta^2$ are indexed by 
$t_1^{(j_1)}=q/\beta +j_1/\beta^2$ where $j_1\in \{0,\dots, q\}$. The distance between two such consecutive  numbers is $1/\beta^2$.  

For the interval between $q/\beta +\cdots +q/\beta^{n-1}$ and $1$ we let  $j_{n-1}\in\{0,\dots,q\}$ and define $t_{n-1}^{(j_{n-1})}:=q/\beta +\cdots +q/\beta^{n-1}+j_{n-1}/\beta^n$. We also have the identities $t_k^{(q)}=t_{k+1}^{(0)}$ when $0\leq k\leq n-1$, and $t_{n-1}^{(q)}=1.$

The distance between two consecutive points depends on which \virg{red} interval they are situated and is given by:
$$t_{k_1}^{(j_{k_1}+1)}-t_{k_1}^{(j_{k_1})}=\beta^{-(k_1+1)},\quad 0\leq k_1\leq n-1.$$

By definition, the {\it first layer} means the set of all numbers $t_{k_1}^{(j_{k_1})}$ where $k_1\in \{0,\dots, n-1\}$ and $j_{k_1}\in \{0,\dots, q\}$. 

At this point we are able to further refine any interval between two consecutive elements of the first layer, where the endpoints $0$ and $1$ are replaced by $t_{k_1}^{(j_{k_1})}$ and $t_{k_1}^{(j_{k_1}+1)}$, and the width $1$ is replaced by $\beta^{-k_1-1}$. More precisely, the points of the \emph{second layer} are defined for $0\leq k_1,k_2\leq n-1$: 
\begin{align*}
t_{k_1\, ,\, k_2}^{(j_{k_1}\, ,\, j_{k_2})}=t_{k_1}^{(j_{k_1})}+\beta^{-(k_1+1)}t_{k_2}^{(j_{k_2})}
\end{align*}
Thus, in particular we have that
\[
t_{k_1}^{(j_{k_1})}\leq  t_{k_1\, ,\, k_2}^{(j_{k_1}\, ,\, j_{k_2})}\leq t_{k_1}^{(j_{k_1}+1)},\quad
t_{k_1\, ,\, q}^{(j_{k_1}\, ,\, n-1)}=t_{k_1}^{(j_{k_1}+1)}.
\]
In general, the \emph{$m$'th layer} consists of the points for $0\leq k_1,k_2,\dots,k_m\leq n-1$:
$$t_{k_1,k_2,\dots,k_m}^{(j_{k_1},\, j_{k_2},\dots, j_{k_m})}=t_{k_1}^{(j_{k_1})}+\beta^{-(k_1+1)}t_{k_2}^{(j_{k_2})}+\dots +\beta^{-(k_1+1)}\cdots \beta ^{-(k_{m-1}+1)}t_{k_m}^{(j_{k_m})}.$$

We now introduce the $L^1$ normalized indicator functions of intervals between two \virg{consecutive points} of layer $m$ denoted by:
\begin{equation}\label{11}
F_{k_1,k_2,\dots, k_m}^{(j_{k_1},\, j_{k_2},\dots, j_{k_m})}(x) =\beta^{k_1+1}\dots \beta^{k_{m}+1}\, \chi_{\big [t_{k_1,k_2,\dots,k_m}^{(j_{k_1},\, j_{k_2},\dots, j_{k_m})}\, , \, t_{k_1,k_2,\dots,k_m}^{(j_{k_1},\, j_{k_2},\dots, j_{k_m}+1)}\big ]}(x).
\end{equation}

Finally, let us introduce a special notation for the red numbers 
including the endpoints $0$ and $1$. They are:   $$t_0:=t_0^{(0)}=0,\quad  t_1:=t_0^{(q)}=t_1^{(0)}=q/\beta,\quad t_2:=t_1^{(q)}=t_2^{(0)}=q/\beta +q/\beta^2,\quad \dots,$$ $$t_{n-1}:=t_{n-2}^{(q)}=t_{n-1}^{(0)}=q/\beta +\cdots +q/\beta^{n-1},\quad  \text{and} \quad t_n:=t_{n-1}^{(q)}=1.$$ The two very last notations give the $L^1$ normalized indicator functions of the intervals between two such consecutive points: 
\begin{equation}\label{10}
F_r(x):=q^{-1}\sum_{j=0}^{q-1} F_r^{(j)}(x)= q^{-1}\beta^{r+1}\chi_{[t_r,t_{r+1}]}(x),\quad 0\leq r\leq n-1.
\end{equation}

\begin{lemma}\label{lemmagh1}
We have  $$\T F_0=\chi_{[0,1]}=q\sum_{j=0}^{n-1} \beta^{-(j+1)} F_{j},\,\text{and}\quad   \T F_r=F_{r-1}\quad \text{where}\quad  1\leq r\leq n-1.$$ In particular, the subspace generated by these functions is invariant under the action of $\T$, namely  $\T \left( {\rm span}\{F_0,\dots,F_{n-1}\}\right)\subseteq {\rm span}\{F_0,\dots,F_{n-1}\}. $ 

Moreover, for all $m\geq 2$ and all possible tuples  $(j_{k_1},\, j_{k_2},\dots, j_{k_m})\in \{{0,\dots,q\}}^m$ we have:  
\begin{equation}\label{may100}
\T F_{k_1,k_2,\dots,k_m}^{(j_{k_1},\, j_{k_2},\dots, j_{k_m})}=F_{k_2,\dots,k_m}^{(j_{k_2},\dots, j_{k_m})} \quad \text{if}\quad k_1=0,
\end{equation} 
\begin{equation}\label{may101}
\T F_{k_1,\, k_2,\dots,k_m}^{(j_{k_1},\, j_{k_2},\dots, j_{k_m})}=F_{k_1-1,k_2,\dots,k_m}^{(j_{k_1},\, j_{k_2},\dots, j_{k_m})} \quad \text{if}\quad k_1\geq 1,
\end{equation} 
and 
\begin{equation}\label{may10}
\T^{m-1+k_1+k_2+\cdots +k_{m-1}} F_{k_1,\dots, k_m}^{(j_{k_1},\, j_{k_2},\dots, j_{k_m})}\in {\rm span}\{F_0,F_1, \dots,F_{n-1}\}.
\end{equation}
\end{lemma}
\begin{proof}
For $x\in [0,1]$ we have
\begin{align*}
&\chi_{\big [t_{k_1,\dots,k_m}^{(j_{k_1},\, j_{k_2},\dots, j_{k_m})}\, , \, t_{k_1,\dots,k_m}^{(j_{k_1},\, j_{k_2},\dots, j_{k_m}+1)}\big ]}\big ((x+j)/\beta\big )\\ 
&\qquad \qquad =\chi_{[0,1]}(x)\, \chi_{\big [\beta\, t_{k_1,\dots,k_m}^{(j_{k_1},\, j_{k_2},\dots, j_{k_m})}-j\, , \, \beta\, t_{k_1,\dots,k_m}^{(j_{k_1},\, j_{k_2},\dots, j_{k_m}+1)}-j\big ]}(x),
\end{align*}
which introduced in \eqref{3} gives for the functions $F_{k_1, k_2,\dots,k_m}^{(j_{k_1},\, j_{k_2},\dots, j_{k_m})}$ defined in \eqref{11}:
\begin{equation}\label{gh1}
\begin{aligned}
 &(\T F_{k_1,k_2,\dots,k_m}^{(j_{k_1},\, j_{k_2},\dots, j_{k_m})})(x)=\beta^{-1}\beta^{k_1+1}\dots \beta^{k_{m}+1}\, \times \\
 &\quad \times \, \chi_{[0,1]}(x)\,\sum_{j=0}^{q}\chi_{\big [\beta\, t_{k_1,\dots,k_m}^{(j_{k_1},\, j_{k_2},\dots, j_{k_m})}-j\, , \, \beta\, t_{k_1,\dots,k_m}^{(j_{k_1},\, j_{k_2},\dots, j_{k_m}+1)}-j\big ]}(x).
\end{aligned}
\end{equation}
{\bf First consider $m=1$.} We start by computing $\T F_0^{(j_0)}$, thus we put $m=1$ and $k_1=0$. Then $\beta t_0^{(j_0)}=j_0\in\{0,\dots,q-1\}$ and 
$$\chi_{[\beta t_0^{(j_0)}-j,\beta t_0^{(j_0+1)}-j]}(x)=\chi_{[j_0-j,j_0-j+1]}(x).$$
By summing over $j$ in  \eqref{gh1} we get 
$$\T F_0^{(j_0)}=\chi_{[0,1]},\quad 0\leq j_0\leq q-1.$$
Since the above formula is independent of $j_0$, it also implies that $\T F_0=\chi_{[0,1]}$, see \eqref{10} for the definition of $F_0$. 

We now want to compute $\T F_{k_1}^{(j_{k_1})}$ with $0<k_1\leq n-1$. Since $k_1\geq 1$ then $\beta t_{k_1}^{(j_{k_1})}\geq q$ and so the interval $[\beta t_{k_1}^{(j_{k_1})}-j, \beta t_{k_1}^{(j_{k_1}+1)}-j]$ is disjoint from $[0,1]$ if $j\leq q-1$. On the other hand, since 
$$t_{k_1}^{(j_{k_1})}=q/\beta+\cdots +q/\beta^{k_1}+j_{k_1}/\beta^{k_1+1}$$
we have  

$$0\leq \beta t_{k_1}^{(j_{k_1})}-q=t_{k_1-1}^{(j_{k_1})}<t_{k_1-1}^{(j_{k_1}+1)}=\beta t_{k_1}^{(j_{k_1}+1)}-q\leq 1.$$ 
This implies that 
$$\T F_{k_1}^{(j_{k_1})}=F_{k_1-1}^{(j_{k_1})},\quad 1\leq k_1\leq n-1,\quad 0\leq j_{k_1}\leq q-1.$$
This shows that $\T^{1+k_1}F_{k_1}^{(j_{k_1})}=\T F_0^{(j_{k_1})}=\chi_{[0,1]}$ belongs to the subspace spanned by $F_0,\dots, F_{n-1}$ (see \eqref{10}). 
Applying $\T$ to \eqref{10} we obtain
$$\T F_r=F_{r-1},\quad 1\leq r\leq n-1.$$
This ends the proof of the first part of the lemma. 

{\bf Now let us consider $m>1$, i.e.\ more than just one layer}.  
\begin{itemize}
\item If $k_1=0$ then 
\begin{align*}\beta\, t_{0,\dots,k_m}^{(j_{0},\, j_{k_2},\dots, j_{k_m})}-j&= \beta(j_0/\beta+\beta^{-1} t_{k_2}^{(j_{k_2})}+\cdots +\beta^{-1}\cdots \beta^{-(k_{m-1}+1)}t_{k_m}^{(j_{k_m})})-j\\
&=j_0-j+t_{k_2,\dots,k_m}^{(j_{k_2},\dots, j_{k_m})}
\end{align*}
which introduced in \eqref{gh1} gives:
$$\T F_{0,k_2,\dots,k_m}^{(j_{0},\, j_{k_2},\dots, j_{k_m})}=F_{k_2,\dots,k_m}^{(j_{k_2},\dots, j_{k_m})}.$$
This shows that if we apply $\T$ on a function with $k_1=0$, then we go down to a lower layer where $m$ is replaced by $m-1$ and $j_0$ is \virg{erased}. This proves \eqref{may100}. 
\item If $1\leq k_1\leq n-1$ then $\beta\, t_{k_1,\dots,k_m}^{(j_{k_1},\, j_{k_2},\dots, j_{k_m})}\geq q$ and so the sum over $j\leq q-1$ in \eqref{gh1} equals zero. On the other hand, 
\begin{align*}
0\leq \beta t_{k_1,\, k_2,\dots,k_m}^{(j_{k_1},\, j_{k_2},\dots, j_{k_m})}-q&=t_{k_1-1,k_2,\dots,k_m}^{(j_{k_1},\, j_{k_2},\dots, j_{k_m})}<t_{k_1-1,k_2,\dots,k_m}^{(j_{k_1},\, j_{k_2},\dots, j_{k_m}+1)}\\
&=\beta t_{k_1,\, k_2,\dots,k_m}^{(j_{k_1},\, j_{k_2},\dots, j_{k_m}+1)}-q\leq 1,\end{align*}
hence 
$$\T F_{k_1,\, k_2,\dots,k_m}^{(j_{k_1},\, j_{k_2},\dots, j_{k_m})}=F_{k_1-1,k_2,\dots,k_m}^{(j_{k_1},\, j_{k_2},\dots, j_{k_m})}.$$
This shows that when we apply $\T$ on a function of the type \eqref{11} with $k_1>0$, then $k_1$ is reduced with one unit. This proves \eqref{may101}. 
\end{itemize}
Conclusion: it takes  $k_1+1$ applications of $\T$ in order to go down from layer $m$ to layer $m-1$, then $k_2+1$ applications in order to get from layer $m-1$ to layer $m-2$, so $\T^{k_1+k_2+\cdots k_{m-1}+m-1}$ gets us to the lowest layer with $m=1$.
\end{proof}

\subsubsection{Proof of Proposition \ref{prop:u1}} 

\begin{lemma}\label{lemmagh3}
Denote by $\mathcal{T}$ the $n\times n$ matrix obtained by restricting $\T$ to the subspace generated by $\{F_0,\dots,F_{n-1}\}$. Then $\mathcal{T}$ is a left-stochastic matrix. If $\lambda$ is an eigenvalue, then it obeys the equation $P_{n,q}(\lambda \beta)=0$ with $P_{n,q}$ from Lemma \ref{lemmagh2}. For $\lambda_1=1$ we can construct a positive eigenvector. If $\lambda_2$ is the second largest eigenvalue in absolute value, then 
\begin{equation}\label{gh10}
q^{1/(n-1)}\beta^{-n/(n-1)}\leq |\lambda_2|<\beta^{-1}.
\end{equation}
There exists an explicitly computable piecewise constant function $u_1$ which is positive a.e.\ such that 
  \begin{equation}\label{dhc2}
     \T u_1=u_1,\quad   u_1\in {\rm span}\{F_0,\dots,F_{n-1}\},\quad \int_0^1 u_1(x)dx=1\, .
  \end{equation}
Moreover, there exists $C<\infty $ such that for every $r\in \mathbb{N}$ and any $g\in {\rm span}\{F_0,\dots,F_{n-1}\}$  we have: 
\begin{equation}\label{dhc3}
      \Big \Vert (\T^r g)(\cdot )-u_1(\cdot) \int_0^1 g(t)dt \Big \Vert_{L^\infty}\leq C\, |\lambda_2|^r \|g\|_{L^1}.
  \end{equation}
  
\end{lemma}
\begin{proof}
    We have 
    \begin{align*}
    \T F_{j-1} =\sum_{i=1}^n \mathcal{T}_{ij} \, F_{i-1},\quad 1\leq j\leq n,\qquad 
\mathcal{T}=\begin{bmatrix} q\beta^{-1} & 1 & 0& \dots &0 &0 \\
q\beta^{-2} & 0 & 1 &\dots &0&0 \\
\vdots &\vdots&\vdots&\vdots&\vdots&\vdots   \\
q\beta^{-(n-1)} & 0 & 0 &\dots &0&1\\
q\beta^{-n} & 0 & 0 &\dots &0&0 
\end{bmatrix},
    \end{align*}
then $\mathcal{T}$ is left-stochastic by \eqref{1}. Observe that
    \begin{align*}
z\, {\rm Id}_n-\mathcal{T}=\begin{bmatrix} z-q\beta^{-1} & -1 & 0& \dots &0 &0 \\
-q\beta^{-2} & z & -1 &\dots &0&0 \\
\vdots &\vdots&\vdots&\vdots&\vdots&\vdots  \\
-q\beta^{-(n-1)} & 0 & 0 &\dots &z&-1\\
-q\beta^{-n} & 0 & 0 &\dots &0&z
\end{bmatrix}.
    \end{align*}
    Expanding the determinant with respect to the first row we get 
    $${\rm det}\big (z\, {\rm Id}_n-\mathcal{T}\big )=(z-q\beta^{-1})z^{n-1}+ {\rm det}(\mathcal{T}_{n-1})$$
    where 
\begin{align*}
\mathcal{T}_{n-1}=\begin{bmatrix} 
-q\beta^{-2} & -1 &\dots &0&0 \\
-q\beta^{-3} & z &-1& \dots &0\\
\vdots &\vdots&\vdots&\vdots&\vdots \\
-q\beta^{-(n-1)} & 0 &\dots &z&-1\\
-q\beta^{-n} & 0 &\dots &0&z
\end{bmatrix}.
    \end{align*}
    By recursion we get 
    $${\rm det}\big (z\, {\rm Id}_n-\mathcal{T}\big )=(z-q\beta^{-1})z^{n-1}-q\beta^{-2}z^{n-2}-\cdots -q\beta^{-(n-1)}z-q\beta^{-n}=\beta^{-n}P_{n,q}(z\beta).$$
    Thus $\lambda$ is an eigenvalue if and only if $\lambda \beta$ is a zero of $P_{n,q}$, hence all eigenvalues are simple due to Lemma \ref{lemmagh2}\ref{itB:1} and \ref{itB:3}. While $\lambda_1=1$ (notice that $\lambda_1=1$ is an eigenvalue due to \eqref{1}), all other eigenvalues are in absolute value less than $\beta^{-1}<1$ due to Lemma \ref{lemmagh2}(iii). Since the product of all roots of $P_{n,q}$ must equal $(-1)^{n-1}q$, we have 
    $$\beta\,|\beta \lambda_2|\cdots |\beta \lambda_n|=q. $$
    If $\lambda_2$ has the second largest modulus, we have 
    $q\leq \beta^n |\lambda_2|^{n-1}$,
    which proves the lower bound in \eqref{gh10}. 

Now let us compute an eigenfunction corresponding to the eigenvalue $1$. We solve the system 
\begin{align*}
\begin{bmatrix} 1-q\beta^{-1} & -1 & 0& \dots &0 &0 \\
-q\beta^{-2} & 1 & -1 &\dots &0&0 \\
\vdots &\vdots&\vdots&\vdots&\vdots&\vdots  \\
-q\beta^{-(n-1)} & 0 & 0 &\dots &1&-1\\
-q\beta^{-n} & 0 & 0 &\dots &0&1
\end{bmatrix} \begin{bmatrix} s_1\\ s_2\\ \vdots \\ s_{n-1}\\ s_n\end{bmatrix}=\begin{bmatrix} 0\\ 0\\ \vdots \\ 0\\ 0\end{bmatrix}.
    \end{align*}
    We may choose $s_1$ as a free variable. In that case we may choose: 
    \begin{align*}
    s_1&=1,\\
    s_2&=1-q\beta^{-1},\\ 
    s_3&=1-q\beta^{-1}-q\beta^{-2},\\
    &... \\
    s_n&=1-q\beta^{-1}-\cdots -q\beta^{n-1}=q\beta^{-n}.
    \end{align*}
    Now let us define (see \eqref{10}) $\tilde{F}_k(x)=\sqrt{q}\, \beta^{-(k+1)/2}F_k(x)$ for $0\leq k\leq n-1$. They form an $L^2$-orthonormal basis in the span of $\{F_0,\dots, F_{n-1}\}$. The restriction of $\T$ to this subspace, in the new basis, will have a matrix (here $1\leq i,j\leq n$)
    \begin{align*}
\widetilde{\mathcal{T}}_{ij}&:=\langle \tilde{F}_{i-1},\T\tilde{F}_{j-1}\rangle = \sqrt{q}\,\beta^{-j/2}\langle \tilde{F}_{i-1},\T{F}_{j-1}\rangle =\sqrt{q}\,\beta^{-j/2}\sum_{r=1}^n\mathcal{T}_{rj}\, \langle \tilde{F}_{i-1},F_{r-1}\rangle\\
&=\beta^{i/2}\, \mathcal{T}_{ij}\, \beta^{-j/2}\,.  
\end{align*}
Since $\mathcal{T}$ and $\widetilde{\mathcal{T}}$ are similar, $\widetilde{\mathcal{T}}$ has the same spectrum as $\mathcal{T}$. Moreover, the vector $\tilde{s}$ with coordinates $\tilde{s}_j=\beta^{j/2}s_j$, where $1\leq j\leq n$, is a not-normalized eigenvector of $\widetilde{\mathcal{T}}$ corresponding to the eigenvalue $1$. The adjoint matrix $\widetilde{\mathcal{T}}\, ^*$ has the matrix elements 
$$\Big (\widetilde{\mathcal{T}}^*\Big )_{ij}=\widetilde{\mathcal{T}}_{ji}=\beta^{j/2}\, \mathcal{T}_{ji}\, \beta^{-i/2}.$$
  By direct computation, using that $\sum_{j=1}^n\mathcal{T}_{ji}=1$ for all $i$, we can check that the vector $\tilde{t}$ with entries $\tilde{t}_j=\beta^{-j/2}$ is an eigenvector of $\widetilde{\mathcal{T}}^*$ corresponding to the same eigenvalue $1$.  
  
  Getting back to functions, the operator $\T$ has an eigenfunction $u(x)$ corresponding to eigenvalue $1$ given a.e.\ by 
  $$u(x)=\sum_{j=1}^n \tilde{s}_j \, \tilde{F}_{j-1}(x)=\sqrt{q}\sum_{j=1}^ns_jF_{j-1}(x)>0,$$
  and we denote by 
  \begin{equation*}
      u_1(x):=\frac{u(x)}{\int_0^1 u(t)dt},\quad \int_0^1 u_1(x)dx=1,
  \end{equation*}
  which satisfies \eqref{dhc2}. 
  
  Using the information we have about the eigenvector $\tilde{t}$ of $\widetilde{\mathcal{T}}^*$, the adjoint $\T^*$ of $\T$ seen as an operator on the span of $\{F_0,\dots,F_{n-1}\}$ has an eigenfunction 
  $$w(x)=\sum_{j=1}^n \tilde{t}_j\, \tilde{F}_{j-1}(x)=\sqrt{q}\, \sum_{j=1}^n \beta^{-j} F_{j-1}(x)=q^{-1/2}\, \chi_{[0,1]}(x),\quad \T^*\chi_{[0,1]}=\chi_{[0,1]}.$$
  Then the rank-one Riesz projection corresponding to the eigenvalue $1$ can be written as 
  $$\Pi_1=|u_1\rangle \langle \chi_{[0,1]}|,\quad \Pi_1^2=\Pi_1.$$
  Moreover, we may write 
  $$\T|_{{\rm span}\{F_0,\dots,F_{n-1}\}}=\Pi_1+\sum_{j=2}^n \lambda_j \Pi_j$$
  where each projection has rank one and $\Pi_j\Pi_k=\delta_{jk}\Pi_k$. Now if $g$ is in the span of $\{F_0,\dots,F_{n-1}\}$ we have 
  $$\T^r g= u_1 \int_0^1 g(t)dt +\sum_{j=2}^n \lambda_j^r \, \Pi_j g.$$
  Since each $\Pi_j$ is a rank one operator of the form 
  $$ \big (1/\langle v_j, u_j\rangle_{L^2}\big )\, |u_j\rangle \, \langle v_j| $$
  with $u_j$ and $v_j$ bounded functions in the span of $\{F_0,\dots,F_{n-1}\}$, we have 
  $$\Vert \Pi_j g\Vert_{L^\infty}\leq C\, \Vert g\Vert_{L^1},\quad 2\leq j\leq n.$$
\end{proof}

\subsubsection{Finalizing the proof of Theorem \ref{thm1}} 
The first step is to approximate $f$ with piecewise constant functions using its Lipschitz property. For example, using the first layer in Figure \ref{fig:1} we have (in the sup-norm)
$$f-\sum_{k_1=0}^{n-1} \sum_{j_{k_1}=0}^{q-1}f\big (t_{k_1}^{(j_{k_1})}\big ) \, \beta^{-1-k_1}\, F_{k_1}^{(j_{k_1})}= \mathcal{O}( L_f\,\beta^{-1}),$$
where $F_{k_1}^{(j_{k_1})}$ is defined in \eqref{11}.
The error is largest on the interval between $0$ and $q/\beta$, because the distance between two consecutive points is only $\beta^{-1}$. On the other intervals, where $k_1\geq 1$, the distance between two consecutive points is at least $\beta^{-2}$ and the error is of order $\beta^{-2}$ or better. 

It is possible to improve the above estimate and get a global error of order $\beta^{-2}$. To achieve this, we have to refine the interval $[0,q/\beta]$ by going to the second layer, while keeping unchanged the other intervals where $k_1\geq 1$. This leads to:
\begin{align*}
f- \sum_{j_{0}=0}^{q-1}\sum_{k_2=0}^{n-1}\sum_{j_{k_2}=0}^{q-1}f\big (t_{0,k_2}^{(j_0, j_{k_2})}\big ) \, \beta^{-2-k_2}\, F_{0,k_2}^{(j_{0}, j_{k_2})}-\sum_{k_1=1}^{n-1} \sum_{j_{k_1}=0}^{q-1}f\big (t_{k_1}^{(j_{k_1})}\big ) \, \beta^{-1-k_1}\, F_{k_1}^{(j_{k_1})}= \mathcal{O}(L_f \beta^{-2}).
\end{align*}

If we want a global error of order $\beta^{-3}$, we need to go up to the third layer on the subintervals where $k_2=0$ in the triple sum, and to the second layer  on the subintervals where $k_1=1$ in the double sum.  

If we want a global error of order $\beta^{-n-1}$, even the old subinterval $[1-q/\beta^n,1]$ corresponding to $k_1=n-1$ in the first layer has now to be refined with a second layer. 

In the general case, let us fix some integer $M\geq n+1$ and let us investigate in which way we should split the interval $[0,1]$ so that the error we make is not bigger than $\beta^{-M+n}$. From the above discussion, this amounts to adjust the length of the subintervals obtained by picking points from different layers.  

For a given layer of order $m\geq 1$, the support of $F_{k_1,\dots,k_m}^{(j_{k_1},\,\dots, j_{k_m})}$ has a width of $\beta^{-m-k_1-\cdots -k_m}$. 
We have the following double inequality: 
\begin{equation}\label{may55}
\begin{aligned}
 k_1+k_2+\cdots +k_m+m &<k_1+k_2+\cdots +k_m +k_{m+1}+(m+1) \\
 &\leq k_1+k_2+...+k_m+m +n, 
\end{aligned}
\end{equation}
where the first one is trivial while the second one is due to $k_{m+1}\leq n-1$. 

Remember that $M\geq n+1$. The first layer has $m=1$ with $k_1+1<M$ because $k_1\leq n-1$. By refining each subinterval of layer $1$ by adding points of higher layers, we  have two alternatives: 
\begin{itemize}
\item Either: $$k_1+k_2+\cdots +k_m +k_{m+1}+(m+1)<M$$
\item Or: $$k_1+k_2+\cdots +k_m+m <M\leq k_1+k_2+\cdots +k_m +k_{m+1}+(m+1).$$
\end{itemize}
If the first alternative is realized, then we perform another refinement. If the second alternative is realized, (this must happen at some point), then by coupling it with \eqref{may55} we obtain
\begin{equation}\label{dhc10}
k_1+k_2+...+k_m+m< M\leq n+k_1+k_2+...+k_m+m .
\end{equation}
No further refinement is performed on a subinterval where \eqref{dhc10} holds. Also, when \eqref{dhc10} is satisfied, we write
$$m+k_1+\cdots +k_m\approx M.$$
Replacing $f$ on the support of $\chi_{\big [t_{k_1,\dots,k_m}^{(j_{k_1},\,\dots, j_{k_m})}, t_{k_1,\dots,k_m}^{(j_{k_1},\, \dots, j_{k_m}+1)}\big ]}$ with $f(t_{k_1,\dots,k_m}^{(j_{k_1},\,\dots, j_{k_m})})$  and using the Lipschitz property of $f$, the error is of order  $\beta^{-m-k_1-\cdots -k_m}$. Thus we have 
(even in the sup-norm) 
$$f- \sum_{\underset{j_{k_1},\dots,j_{k_m}}{m+k_1+\cdots +k_m\approx M}}\, f(t_{k_1,\dots,k_m}^{(j_{k_1},\, \dots, j_{k_m})})\beta^{-m-k_1-\cdots -k_m}\, F_{k_1,\dots,k_m}^{(j_{k_1},\, \dots, j_{k_m})}=\mathcal{O}(L_f \beta^{-M}).$$
According to \eqref{eqn:Pnonexp}, $\T$ is a non-expansive map on $L^1$,  hence there exists a constant $C<\infty$ such that for all $N\geq 1$ we have:
$$\Big \Vert \T^Nf-
 \sum_{\underset{j_{k_1},\dots,j_{k_m}}{m+k_1+\cdots +k_m\approx M}}\, f(t_{k_1,\dots,k_m}^{(j_{k_1},\,\dots, j_{k_m})})\beta^{-m-k_1-\cdots -k_m}\, \T^N F_{k_1,\dots,k_m}^{(j_{k_1},\,\dots, j_{k_m})}\Big \Vert_{L^1}\leq C\,L_f \beta^{-M}.$$  If $N$ is larger than $M$, which is already larger than $m+k_1+\cdots+k_m$ (due to \eqref{dhc10}), then according to \eqref{may10} in Lemma \ref{lemmagh1} we have that both $\T^{N}F_{k_1,\dots,k_m}^{(j_{k_1},\, \dots, j_{k_m})}$ and $\T^{ M}F_{k_1,\dots,k_m}^{(j_{k_1},\, \dots, j_{k_m})}$  belong to the invariant subspace, are non-negative and their $L^1$ norm is constant equal to $1$ due to \eqref{eqn:intT}. Using \eqref{dhc3} with $r=N-M$ we have that in the $L^1$ sense:
 $$(\T^Nf)(\cdot)- \sum_{\underset{j_{k_1},\dots,j_{k_m}}{m+k_1+\cdots +k_m\approx M}}f\big (t_{k_1,\dots,k_m}^{(j_{k_1},\, \dots, j_{k_m})}\big )\,  \beta^{-m-k_1-\cdots-k_m} \Big (u_1(\cdot )+\mathcal{O}(|\lambda_2|^{N-M})\Big )=\mathcal{O}(L_f \beta^{-M}),$$
 where the bounding constants appearing in the two errors are independent of $N$ and $M$. 
 Up to another error of order $\mathcal{O}(\beta^{-M})$ we may replace the Riemann sum with $\int_0^1f(t)dt$. Hence, we have
 $$\T^Nf-u_1\int_0^1 f(t)dt =\mathcal{O}( \Vert f\Vert_{L^\infty}\, |\lambda_2|^{N-M})+\mathcal{O}(L_f\,  \beta^{-M}),\quad N>M.$$
 Given $N\gg 1$, we may choose an \virg{optimal} $M$ as a function of $N$ such that 
 $$|\lambda_2|^{N-M}\sim\beta^{-M},$$
 where $\sim$ means that they may differ by a numerical factor which is independent on $N$. If $n=2$ then $|\lambda_2|=q\beta^{-2}$, hence we may choose $M$ to be the integer part of $x$ where $x$ solves the equation 
$$x\, \ln(\beta)=(N-x)\ln (\beta^2/q),$$
 which gives $x=K_2 N$ with $K_2$ in \eqref{aug1}. 
 
  Also, since $|\lambda_2|<1/\beta$ for all $n$ (see \eqref{gh10}), by choosing $M$ to be the integer part of $N/2$ we see that the decay is always faster than $\beta^{-N/2}$. \qed

\section{Proof of Theorem \ref{thm2}}\label{sect3}

\subsection{Proof of \ref{itT2:1}.}
We only prove the result for $q>1$. Let us first show that $j/\beta<x_j<(j+1)/\beta$ for all $0\leq j\leq q-1$. The first inequality follows directly from the definition of $x_j$, while the second one is equivalent with 
 $$q\beta^{-2}+\cdots + q\beta^{-n}<\beta^{-1} \quad \text{or}\quad q\beta^{-1}+\cdots + q\beta^{-(n-1)}<1,$$
 the latter holds true by \eqref{1}.
This shows that $\psi_0$ is well defined on $[0,1]$  and by convention, it equals zero outside this interval. 

In view of \eqref{3}, if $q\beta^{-1}+\cdots + q\beta^{-(n-1)}<x<1$ we have 
$$(\T \psi_0)(x)=\beta^{-1}\sum_{j=0}^{q-1} \psi_0\big ((x+j)/\beta\big ).$$
For $x$ in that interval we also have 
$$q\beta^{-2}+\cdots + q\beta^{-n}+j/\beta =x_j<(x+j)/\beta<(j+1)/\beta,\quad 0\leq j\leq q-1, $$
which from the definition of $\psi_0$ it implies 
$$(\T \psi_0)(x)=\beta^{-1}\sum_{j=0}^{q-1} e^{2\pi i  (x+j)/q}=\beta^{-1}e^{2\pi ix/q}\sum_{j=0}^{q-1} \Big (e^{2\pi i/q}\Big )^j=0.$$
If $0<x<q\beta^{-1}+\cdots + q\beta^{-(n-1)}$ we have 
$$(\T \psi_0)(x)=\beta^{-1}\sum_{k=0}^{q} \psi_0\big ((x+k)/\beta\big ).$$
For $x$ in the above interval we also have 
$$ k/\beta < (x+k)/\beta < q/\beta^2 + \cdots + q/\beta^n + k/\beta = x_k, \quad 0 \le k \le  q,$$ 
which from the definition of $\psi_0$ it implies 
$$(\T \psi_0)(x)=\beta^{-1}\sum_{k=0}^{q} e^{2\pi i  (x+k)/(q+1)}=\beta^{-1}e^{2\pi ix/(q+1)}\sum_{k=0}^{q} \Big (e^{2\pi i/(q+1)}\Big )^k=0.$$

\subsection{Proof of \ref{itT2:2} and \ref{itT2:3}.} 
Let us show that $\widetilde{\K}=u_1^{1/p}\, \K \, \frac{1}{u_1^{1/p}}$ is an isometry on $L^p([0,1])$.  If $p=\infty$ then this follows directly from the definition in \eqref{may1}. If $1\leq p<\infty$ we have (using \eqref{oct1} in the third equality):
$$\Vert \widetilde{\K}(f)\Vert_{L^p}^p=\int_0^1 |\widetilde{\K}(f)|^pdx=\int_0^1 u_1 \K(|f|^p/u_1)\, dx=\int_0^1 (\T u_1) |f|^p/u_1\, dx=\Vert f\Vert_{L^p}^p.$$

The operator $\widetilde{\K}-z\,{\rm Id}  =u_1^{1/p}\, ( \K-z\, {\rm Id})\, u_1^{-1/p}$ is invertible if and only if $\K -z\, {\rm Id}$ is invertible,  hence $\widetilde{\K}$ and $\K$ have the same spectrum. Since $\widetilde{\K}$ is an isometry, it is also injective, hence $\K$ is injective, too. 

Now let us show that $\K$ (thus also $\widetilde{\K}$) is not surjective.  Using \eqref{oct1} and the eigenvector $\psi_0$ of $\T$ constructed at point (i) ($\psi_0$ belongs to any $L^{p'}$ with $1\leq p'\leq \infty$) we have 
$$\int_0^1 \overline{\psi_0(x)}\, (\K \, g)(x)\, dx=0,\quad \forall g\in L^p,\quad 1\leq p\leq \infty,$$
which implies that $\psi_0$ does not belong to the range of $\K$. This also implies that $u_1^{1/p}\psi_0$ does not belong to the range of $\widetilde{\K}$. 

Thus $\widetilde{\K}$ is a non-surjective isometry and its spectrum must equal the closed unit disk due to the following result which may be found in \cite[Proposition 5.2]{AC}, but we also prove it here (in a more self-contained way) for the convenience of the reader:  
\begin{lemma}\label{lemma-may1}
    Assume that $U$ defined on some Banach space is a linear isometry. If $U$ is surjective, then $\sigma(U)\subset \mathbb{S}^1$. If $U$ is not surjective then $\sigma(U)=\overline{\mathbb{D}}$. 
\end{lemma}
\begin{proof}
An isometry is always injective.  Let us first consider the case when $U$ is surjective (thus invertible). Using that $\Vert Uf\Vert =\Vert f\Vert$ for all $f$ and also $\Vert U^{-1}g\Vert=\Vert U(U^{-1} g)\Vert=\Vert g\Vert$ we conclude that both $U$ and $U^{-1}$ have norm one. Let $z\in \C$ be with $|z|<1$. Then $U-z\, {\rm Id}=( {\rm Id}-zU^{-1})\,U $ is invertible because $\Vert zU^{-1}\Vert <1$. If $|z|>1$ we have 
    $U-z\, {\rm Id}=-({\rm Id}-z^{-1}U)z$ which is also invertible. Thus $\sigma(U)$ is included in the unit circle.  

    Now let us consider the case when $U$ is not surjective. Because $\Vert U\Vert =1$ we know that $\sigma(U)\subset \overline{\mathbb{D}}$. Because $U$ is not invertible, then $0\in \sigma(U)$, hence $\sigma(U)$ has elements which are not on the unit circle. Thus if the inclusion $\sigma(U)\subset \overline{\mathbb{D}}$ is strict, there must exist a point $\lambda$ with $|\lambda|<1$ which belongs to the boundary of $\sigma(U)$. We will now show that $\lambda$ must be in the resolvent set of $U$, which would lead to a contradiction.

Since $\lambda\in \partial \big (\sigma(U)\big )$ there must exist a sequence of points $\lambda_n$ in the resolvent set of $U$ such that $\lambda_n\to \lambda$ when $n\to\infty$.  Since $|\lambda|<1$ there exists $N>1$ such that $|\lambda_n|\leq (1+|\lambda|)/2<1$ if $n>N$. Using the triangle inequality we get
$$\Vert (U-\lambda_n\, {\rm Id})f\Vert \geq \Vert Uf\Vert -|\lambda_n|\, \Vert f\Vert \geq \frac{1-|\lambda|}{2}\, \Vert f\Vert, \quad n>N. $$
Since $U-\lambda_n\, {\rm Id}$ is invertible, using this inequality with $f=(U-\lambda_n\, {\rm Id})^{-1}g$ we obtain: 
$$\Vert (U-\lambda_n\, {\rm Id})^{-1}\Vert \leq \frac{2}{1-|\lambda|},\quad n>N.$$
This uniform bound and the identity
$$U-\lambda\, {\rm Id}= \big ({\rm Id}+(\lambda_n-\lambda)(U-\lambda_n\, {\rm Id})^{-1}\big )\, (U-\lambda_n\, {\rm Id})$$
show that the right hand side must be invertible if $n$ is large enough, hence $\lambda$ is in the resolvent set of $U$ and cannot belong to the boundary of $\sigma(U)$. 
    \end{proof}

\subsection{Proof of \ref{itT2:4}.}

We know from (ii) that $u_1^{1/2}\, \K \, u_1^{-1/2}$ is an isometry on the Hilbert space $L^2([0,1])$.  Then \eqref{oct1} implies that $\T=\K^*$ and:
\begin{align}\label{hcj2}
   \Big (u^{-1/2} \T \, u_1^{1/2}\Big ) \Big (u_1^{1/2}\, \K \, u_1^{-1/2}\Big )=\Big (u^{1/2}\,  \K \, u_1^{-1/2}\Big )^* \Big (u_1^{1/2}\, \K\,  u_1^{-1/2}\Big )= {\rm Id}.
\end{align}

The isometry $u_1^{1/2}\, \K \, u_1^{-1/2}$ has norm one. If $|z|<1$ then $\psi_z$  is different from zero and can be written with the help of a Neumann series. Finally, 
\begin{align*}
\T \psi_z&=\T \psi_0+\sum_{m\geq 1} z^{m} u_1^{1/2}\, \Big (u_1^{-1/2} \T u_1^{1/2}\Big )\, \Big (u_1^{1/2} \K u_1^{-1/2}\Big )^{m} u_1^{-1/2}\psi_0\\
&=\sum_{m\geq 1} z^{m} u_1^{1/2}\, \Big (u_1^{1/2} \K u_1^{-1/2}\Big )^{m-1} u_1^{-1/2}\psi_0=z\psi_z,
\end{align*}
where in the second equality we used $\T \psi_0=0$ and \eqref{hcj2}.



\begin{funding}
This work was funded by the Independent Research Fund Denmark--Natural Sciences,
grant DFF–10.46540/2032-00005B. G. M. gratefully acknowledges financial support from the European Research Council through the ERC CoG UniCoSM, grant agreement n.724939.
\end{funding}


\appendix 

\section{The greedy algorithm}
\label{app:gral}
Let $x\in [0,1)$.  Applying the map $T_\beta$ we get that:
$$
T_{\beta}(x) = \beta x - \lfloor\beta x\rfloor\in [0,1),
$$
where $\lfloor\,\cdot\, \rfloor$ is the floor function and  $q<\beta=\beta_{n,q}<q+1$ in view of Lemma \ref{lemmagh2}\ref{itB:1}.
By iterating the map $T_\beta$, we define the \emph{$j$-th greedy coefficient} as:
\begin{equation}
\label{eqn:defxj}    
x_j := \lfloor\beta T^{(j-1)}_{\beta}(x) \rfloor\quad\forall\, j\geq 1\; \text{ with }\; T_{\beta}^{0}(x) := x.
\end{equation}
The following lemma describes the greedy algorithm.

\begin{lemma}  \label{prop10}
With the definitions above and with $\beta$ as in \eqref{1}, if $x\in [0,1)$ we have 
\begin{equation} \label{basicsum}
x = \sum_{j=1}^\infty x_j \beta^{-j}.
\end{equation}
The scaled remainder $\beta^k(x- \sum_{j=1}^k x_j\beta^{-j})$ obeys 
\begin{equation} \label{scaledremainder}
\beta^k(x- \sum_{j=1}^k x_j\beta^{-j}) = T^k_{\beta}(x).
\end{equation}
 Moreover, the greedy coefficients satisfy three restrictions:
\begin{enumerate}
    \item $x_j \in \{0,1,\cdots,q\}$ for all $j\geq 1$;
    \item $x_j = q $ for $n$ successive $j$'s cannot occur; 
    \item it cannot happen that the sequence of $x_j$'s ends in the infinite sequence $(c_1,c_2,...)$ where $c_{mn}=q-1$ for all $m\geq 1$, and all the other $c_j$'s, with $j$ not dividing $n$, are equal to $q$.
\end{enumerate}
\end{lemma}

\begin{proof}
(\ref{scaledremainder}) is true by definition for $k=1$.  Assuming this equation for some $k \ge 1$ we have $$T_{\beta}^{(k+1)}(x) = T_{\beta}(T_{\beta}^k(x)) = \beta T_{\beta}^k(x) - \lfloor \beta T_\beta^k \rfloor = \beta^{k+1} (x-\sum_{j=1}^k x_j \beta^{-j}) - x_{k+1}$$ $$ = \beta^{k+1}(x - \sum_{j=1}^{k+1} x_j \beta^{-j}).$$
Since $T_{\beta} : [0,1) \rightarrow [0,1)$ and $\beta >1$,  the series in (\ref{basicsum}) converges.

The first restriction on the $x_j$'s follows from their definition:  $0\leq x_j=\lfloor\beta T^{(j-1)}_{\beta}(x) \rfloor\leq \lfloor\beta  \rfloor=q$ because of Lemma \ref{lemmagh2}\ref{itB:1}.

To prove the second restriction on the coefficients suppose that there exists some $k\geq 0$ such that $x_{k+j} = q$, where $j\in\{ 1,..,n\}$. Using \eqref{1} we have 
$$ \sum_{j=1}^n q \beta^{-(k+j)} = \beta^{-k}.$$ If $k=0$, then $x\ge 1$, which is a contradiction. If $k\ge 1 $ then using \eqref{scaledremainder} and \eqref{basicsum} we have
$$T_\beta^k(x)=\beta^k\Big (x-\sum_{j=1}^k x_j \beta^{-j}\Big ) = \beta^k\Big (\sum_{j=k+1}^\infty x_j\beta^{-j}\Big ) \ge \beta^k\Big (\sum_{j=k+1}^{n+k} q\beta^{-j}\Big ) =1 $$
contradicting $T_\beta :[0,1) \rightarrow [0,1).$

In order to prove the third restriction, let us assume that there exists $x\in [0,1)$ whose greedy expansion ends with $\beta^{-k}\sum_{j\geq 1}c_j\beta^{-j}$ for some $k\geq 0$, i.e.\ $x_{k+j}=c_j$ for $j\geq 1$. By repeatedly using \eqref{1} (see also Figure \ref{fig:1}) we have 
\begin{align*}
1&= \sum_{j=1}^{n-1}q\beta^{-j}+(q-1)\beta^{-n} +
\beta^{-n}\\
&=\sum_{j=1}^{n-1}q\beta^{-j}+(q-1)\beta^{-n} +
\beta^{-n}\Big( \sum_{j=1}^{n-1}q\beta^{-j}+(q-1)\beta^{-n}\Big )+\beta^{-2n}=\ldots=\sum_{j\geq 1}c_j\beta^{-j},
\end{align*}
hence $x=\sum_{j=1}^k x_j\beta^{-j} +\beta^{-k}$ and thus by \eqref{scaledremainder} $T_\beta^k(x)=1$, contradiction. 
\end{proof}

\vspace{0.2cm}

Lemma \ref{prop10} has shown that the greedy algorithm gives a unique output for the coefficients $x_j$ defined in \eqref{eqn:defxj} for any number $x\in [0,1)$, and these coefficients obey three necessary conditions. In the next lemma we will show, in particular, that any expansion for $x\in [0,1)$ satisfying all these three conditions must be the greedy one.
\begin{lemma} 
Suppose 
\begin{equation} \label{tilderep}
x = \sum_{j=1}^\infty \tilde {x}_j\beta^{-j}
\end{equation}
where the coefficients $\tilde{x}_j \in \{0,1,\cdots,q\}$ also satisfy the condition that no $n$ consecutive coefficients equal $q$. 
Let $c_j = q-1$ if $n$ divides $j$, and $c_j = q$ otherwise. Let $x_j$ be defined as in \eqref{eqn:defxj}. 
Then one of the following possibilities occurs: 
\begin{enumerate}
\item $\tilde{x}_j = c_j$ for all $j$ in which case $x=1$;  
\item $x<1$ with $\tilde{x}_j = x_j $ for all $j$, i.e.\ $x$ is written in the greedy representation; 
\item $x<1$ and there exists some $k\geq 1$ such that $\tilde{x}_j=x_j$ for $j<k$ (if $k\geq 2$), $\tilde{x}_k=x_k-1$, and $\tilde{x}_{k+j} = c_j$ for $j\geq 1$.  In this case, the finite sum $x = \sum_{j=1}^k x_j \beta^{-j}$ is the greedy representation of x which is different from \eqref{tilderep}.
\end{enumerate}
\end{lemma} 
\begin{proof}
The largest possible value of $\sum_{j=1}^\infty \tilde{x}_j \beta^{-j}$, which can be achieved with the $\tilde{x}_j$ obeying the two restrictions of the current lemma, equals $1$. This is the case if and only if  $\tilde{x}_j = c_j$, for all $j$.  

Assuming $ x < 1$, suppose that the sequence $(\tilde{x}_1,\tilde{x}_2,...)$ does not end in the infinite sequence $(c_1,c_2,...)$ so that the scaled remainder, $\beta^k\sum_{j=k+1}^\infty \tilde{x}_j\beta^{-j}< 1$ for all $k\ge 1$ (we have already assumed this for $k=0$). 
Then $\tilde{x}_j = x_j$ for all $j$:  to see this we have $x_1 = \lfloor \beta x \rfloor$ and  
$\beta x = \tilde{x}_1 + \beta\sum_{j=2}^\infty \tilde{x}_j\beta^{-j} = \tilde{x}_1 + t$ with 
$t\in [0,1)$.  Thus $x_1 = \tilde{x}_1$.  A simple induction gives $x_j = \tilde{x}_j$ for all $j$.

On the other hand, suppose that $k$ is the first integer such that $\beta^k \sum_{j=k+1}^\infty \tilde{x}_j \beta^{-j} = 1.$  Then $\tilde{x}_{k+j} = c_j$, $j\ge 1$ and $\tilde{x}_j =
x_j$, $j<k$, $\tilde{x}_k +1 = x_k\leq q $.  Thus $\tilde{x}_k \leq  q-1$. 
If $\tilde{x}_k = q-1$ the previous (if there are that many) $n-1$ $\tilde{x}_j$'s cannot equal $q$ because that would violate the definition of $k$.  Thus $x = \sum_{j=1}^k x_j \beta^{-j}$, the greedy representation, is a different representation of $x$. \end{proof}

\section{Properties of $\beta_{n,q}$}

The following lemma is given for the sake of the reader and collects in one place a number of known results \cite{B, Mi60}. 
\begin{lemma}\label{lemmagh2}
Let $n,q\in \mathbb{N}$ with $n\geq 2$ and $1\leq q$. Let $$P_{n,q}(z)=z^n-q(z^{n-1}+z^{n-2}+\cdots +z+1)$$ with $z\in \C$.  Then 

\begin{enumerate}[label=(\roman*), ref=(\roman*)]
\item \label{itB:1}
$P_{n,q}$ has only one positive root $\beta_{n,q}$, which also obeys $q<\beta_{n,q}<q+1$.  

\item \label{itB:2} All roots have algebraic multiplicity one. 

\item \label{itB:3} The other roots of $P_{n,q}$ satisfy $\big (q/(q+2)\big )^{1/n}< |z|< 1$. In particular, $\beta_{n,q}$ is a Pisot number.

\item \label{itB:4} Fix $\alpha\in (q,q+1)$. Then there exists $n_0\geq 2$ such that $(q+1)-q\alpha^{-n}\leq \beta_{n,q}<q+1$ for all $n\geq n_0$. 
\end{enumerate}
\end{lemma}
\begin{proof}

{\it \ref{itB:1}} If $x>0$ we define $f(x):=x^{-n}P_{n,q}(x)=1-q(x^{-1}+\cdots+ x^{-n})$. We have that $f'>0$, which means that it can have at most one positive root. 

If $q=1$ we have 
$$f(1)=1- n<0,\quad f(2)=2^{-n}>0$$
hence there exists a unique, simple root between $1$ and $2$. 

For $q>1$ we have
$$f(q)=1-\frac{1-q^{-n}}{1-q^{-1}}=\frac{q^{-n}-q^{-1}}{1-q^{-1}}<0,\quad f(q+1)=1-\frac{q}{q+1}\, \frac{1-(q+1)^{-n}}{1-(q+1)^{-1}}=(q+1)^{-n}>0$$
thus there always exists a unique positive root $\beta_{n,q}\in (q,q+1)$.

{\it \ref{itB:2}} Now let us prove that all the other roots are also simple. If $z\neq 1$ we have $$P_{n,q}(z)=z^n-q\frac{z^n-1}{z-1}=\frac{z^{n+1}-(q+1)z^n +q}{z-1}=: \frac{Q_{n,q}(z)}{z-1}\, .$$

 Since $z=1$ is not a root,  $P_{n,q}(z)$ has the same roots (those different from $1$) as $Q_{n,q}(z)$. If $z_1\neq 1$ is a degenerate root of $P_{n,q}$, i.e. $P_{n,q}(z_1)=P_{n,q}'(z_1)=0$, then we also have  $Q_{n,q}(z_1)=Q_{n,q}'(z_1)=0$. But 
    $$Q_{n,q}'(z)=(n+1)z^n-(q+1)nz^{n-1}=(n+1)z^{n-1}\big ( z-(q+1)n/(n+1) \big )$$
    and since $0$ is not a root, we must have $z_1=(q+1)n/(n+1)$, which is positive. But we know that $P_{n,q}$ only has a non-degenerate positive root, contradiction.

{\it \ref{itB:3}} We want to show that $Q_{n,q}$ has exactly $n$ roots inside the closed unit complex disk. Let $F(z)=z^{n+1}+q$ and $G(z)=-(q+1)z^n$. If $|z|=1+\epsilon$ with $\epsilon>0$ small we have 
$$|F(z)|\leq q+1+(n+1)\epsilon +\mathcal{O}(\epsilon^2),\quad |G(z)|=(q+1)(1+n\epsilon )+\mathcal{O}(\epsilon^2),$$
and since $nq>1$ we have that $|G(z)|>|F(z)|$ on $|z|=1+\epsilon$ if $\epsilon$ is small enough. This implies that the function 
$$H_t(z):=tF(z)+G(z),\quad H_0(z)=G(z),\quad H_1(z)=Q_{n,q}(z)$$
obeys $|H_t(z)|\geq |G(z)|-|F(z)|>0$
on the circle $|z|=1+\epsilon$ for all $t\in [0,1]$. Thus the number of zeros of $H_t$ inside the disk $|z|\leq 1+\epsilon$
is constant in $t$ and equals $n$. Taking the limit $\epsilon\downarrow 0$, we conclude that $Q_{n,q}$ has exactly $n$ zeros inside the complex closed unit disk. Now if $z$ is a zero with $|z|=1$ we have 
$$|z^{n+1}+q|=(q+1)|z^n|=q+1$$
which is possible only for $z^{n+1}=1$. But then $(q+1)z^n=q+1$, hence $z^n=1$. This implies that $z=1$. Hence $P_{n,q}$ has exactly $n-1$ complex roots inside the open unit disk.  

Now let $z_1$ be such a root with $|z_1|<1$ and $Q_{n,q}(z_1)=0$. Then $$(q+1)|z_1|^n=|(q+1)z_1^{n}|\geq q-|z_1|^{n+1}>q-|z_1|^n$$
which leads to 
$$|z_1|^n>q/(q+2).$$

{\it \ref{itB:4}} Fix any $n_0\geq 2$ and let $n\geq n_0$. We have $$1/\beta_{n,q}+\cdots +1/\beta_{n,q}^{n_0}\leq 1/\beta_{n,q}+\cdots+1/\beta_{n,q}^n=1/q,$$ hence $\beta_{n,q}\geq \beta_{n_0,q}$. Also, $Q_{n,q}(\beta_{n,q})=0$, hence $\beta_{n,q}$ solves $\beta_{n,q}=q+1-q/\beta_{n,q}^n$. Thus  
   $$q+1-q/\beta_{n_0,q}^{n}\leq \beta_{n,q}<q+1,\quad 2\leq n_0\leq n.$$
   Now we can choose $n_0$ large enough such that $\beta_{n_0,q}>\alpha$ and we are done. 

\end{proof}


\begin{thebibliography}{00}


 
\bibitem{AA}{A. Abdesselam,} The weakly dependent strong law of large numbers revisited. \emph{G.J.M.} {\bf 3} (2018), no. 2, 94-97  \url{https://gradmath.org/wp-content/uploads/2020/10/Abdesselam-GJM-2018.pdf}

\bibitem{AC}{R.F. Allen, F. Collona,} {Isometries and spectra of multiplication operators on the Bloch space.} \emph{Bull. Aust. Math. Soc.} {\bf 79} (2009), 147–160  \url{https://doi.org/10.1017/S0004972708001196}

\bibitem{Bl}{F. Blanchard,}  
$\beta$-expansions and symbolic dynamics.
\emph{Theoret. Comput. Sci.} {\bf 65}(1989), no. 2, 131–141 \url{https://doi.org/10.1016/0304-3975(89)90038-8}

\bibitem{BG} {A. Boyarsky, P. G{\'o}ra,} \emph{Laws of chaos. Invariant measures and dynamical systems in one dimension.}
Probability and its Applications. Birkh{\"a}user, Boston, MA, 1997. \url{https://link.springer.com/book/10.1007/978-1-4612-2024-4}


\bibitem{B}{A. Brauer,} {On algebraic equations with all but one root in the interior of the unit circle.} \emph{Math. Nachr.} {\bf 4} (1950), 250-257 \url{https://doi.org/10.1002/mana.3210040123}

\bibitem{CCD}{E. Charlier, C.  Cisternino, K. Dajani}, {Dynamical behavior of alternate base expansions.} \emph{Ergod. Th. \& Dynam. Sys.} {\bf 43} (2023), no. 3, 827-860 \url{https://doi.org/10.1017/etds.2021.161}

\bibitem{DK}{K. Dajani, C. Kalle,} \emph{A First Course in Ergodic Theory.} Chapman and
Hall/CRC (2021) \url{https://doi.org/10.1201/9780429276019}

\bibitem{Go}{P. G{\'o}ra,}:
Invariant densities for generalized $\beta$-maps. \emph{Ergod. Th. \& Dynam. Sys.} {\bf 27} (2007), 1583–1598 \url{https://doi.org/10.1017/S0143385707000053}

\bibitem{HMS}{I.W. Herbst, J. M\o ller, A.M. Svane,}
How many digits are needed? \emph{Methodol. Comput. Appl. Probab.} {\bf 26} (2024), article no. 5 \url{https://doi.org/10.1007/s11009-024-10073-2}

\bibitem{HMS2}{I.W. Herbst, J. M\o ller, A.M. Svane,} The asymptotic distribution of the scaled remainder for pseudo golden ratio expansions of a continuous random variable.  \emph{Methodol. Comput. Appl. Probab.}  {\bf 27} (2025), article no. 10,  \url{https://doi.org/10.1007/s11009-025-10137-x}



\bibitem{KLP}{V. Komornik, P. Loreti, M. Pedicini,} A quasi-ergodic approach to non-integer base expansions. \emph{J. Number Theory} {\bf 254} (2024), 146--168 \url{https://doi.org/10.1016/j.jnt.2023.07.009}

\bibitem{LY}{A. Lasota, J.A. Yorke,}
On the existence of invariant measures for piecewise monotonic transformations. \emph{Trans. Amer. Math. Soc.} {\bf 186} (1973), 481–488 \url{https://www.ams.org/journals/tran/1973-186-00/S0002-9947-1973-0335758-1/}

\bibitem{Mi60} {E.P. Miles,} {
Generalized Fibonacci Numbers and Associated Matrices}. \emph{Amer. Math. Monthly} {\bf 67} (1960), no. 8, 745-752 \url{https://doi.org/10.2307/2308649}

\bibitem{Pa}{W. Parry,} On the $\beta$-expansions of real numbers. \emph{Acta Math. Acad. Sci.
Hungar.} {\bf 11} (1960), 401–416 \url{https://doi.org/10.1007/BF02020954}

\bibitem{Pe}{M. Pedicini,} Greedy expansions and sets with deleted digits. \emph{Theoret. Comput. Sci.} {\bf 332} (2005), 313--336  \url{https://doi.org/10.1016/j.tcs.2004.11.002}

\bibitem{Re}{A. R{\'e}nyi,} Representations for real numbers and their ergodic properties. \emph{Acta Math. Acad. Sci.
Hungar.} {\bf 8} (1957), 477–493 \url{https://doi.org/10.1007/BF02020331} 


\bibitem{Suz}{S. Suzuki,} Eigenfunctions of the Perron–Frobenius operators for generalized beta-maps. \emph{Dynamical Systems} {\bf 37} (2022), no. 1, 9–28 \url{https://doi.org/10.1080/14689367.2021.1998378} 

\bibitem{Wal}{P. Walters,} Equilibrium states for $\beta$-transformations and related transformations.  \emph{Math. Z.} {\bf 159}, 65–88 (1978) \url{https://doi.org/10.1007/BF01174569}

\end{thebibliography}
\end{document}